\documentclass[letterpaper, 10 pt, conference]{ieeeconf}  
\IEEEoverridecommandlockouts                        

\usepackage{amsmath,amssymb,amsfonts, mathtools}
\usepackage{algorithmic}
\usepackage{algorithm}
\usepackage{array}
\usepackage[caption=false,font=normalsize,labelfont=sf,textfont=sf]{subfig}
\usepackage{textcomp}
\usepackage{stfloats}
\usepackage{url}
\usepackage{verbatim}
\usepackage{graphicx}
\usepackage{cite}
\usepackage{booktabs}
\usepackage{cancel}
\usepackage[normalem]{ulem}
\usepackage[colorlinks=true, allcolors=blue]{hyperref}

\usepackage{soul}
\usepackage[dvipsnames]{xcolor}

\usepackage{tikz}
\usepackage{pgfplots}
\pgfplotsset{compat=1.18}
\usetikzlibrary{arrows.meta, decorations.pathmorphing, decorations.markings,
                positioning, calc, patterns, shapes.geometric, fadings}
\usetikzlibrary{arrows.meta, calc, fadings}
\newtheorem{proposition}{Proposition}% 
\newtheorem{lemma}{Lemma}
\newtheorem{corollary}{Corollary}

\newtheorem{remark}{Remark}%

\newtheorem{assumption}{Assumption}%

\newcommand{\todo}[1]{{\color{red} [TO-DO] #1}}
\newcommand{\nl}[1]{{\color{black} #1}}
\newcommand{\hs}[1]{{\color{black} #1}}
\newcommand\vs[1]{{\color{black}{#1}}}

\newcommand{\vb}[0]{\boldsymbol{b}}
\newcommand{\va}[0]{\boldsymbol{a}}
\newcommand{\vz}[0]{\boldsymbol{z}}
\newcommand{\vv}[0]{\boldsymbol{v}}
\newcommand{\vw}[0]{\boldsymbol{w}}
\newcommand{\ve}[0]{\boldsymbol{e}}
\newcommand{\vphi}[0]{\boldsymbol{\phi}}
\newcommand{\vpsi}[0]{\boldsymbol{\psi}}
\newcommand{\vzero}[0]{\boldsymbol{0}}

\renewcommand{\theenumi}{\roman{enumi}}

\title{\LARGE \bf
Neuromorphic Realization of Best Response in Finite-Action Games
}

\author{Himani Sinhmar$^{a}$, 
        Vaibhav Srivastava$^{b}$,
        Naomi Ehrich Leonard$^{a}$% <-this % stops a space
    \thanks{}% <-this % stops a space
    \thanks{Research was sponsored by the Army Research Office and was accomplished under Grant Number W911NF2410126. The views and conclusions contained in this document are those of the authors and should not be interpreted as representing the official policies, either expressed or implied, of the Army Research Office or the U.S. Government. The U.S. Government is authorized to reproduce and distribute reprints for Government purposes notwithstanding any copyright notation herein.}
    \thanks{$^{a}$ Department of Mechanical and Aerospace Engineering, Princeton University; \texttt{\{himani.sinhmar,naomi\}@princeton.edu}}%
    \thanks{$^{b}$  Department of Electrical and Computer Engineering, Michigan State University \texttt{vaibhav@msu.edu}}%
}

\begin{document}

\maketitle
\thispagestyle{empty}
\pagestyle{empty}

%%%%%%%%%%%%%%%%%%%%%%%%%%%%%%%%%%%%%%%%%%%%%%%%%%%%%%%%%%%%%%%%%%%%%%%%%%%%%%%%
\begin{abstract}
We develop a mechanistic dynamical-systems formulation of best response in finite-action games with relational structure on the action set. The proposed neuromorphic decision dynamics realize best response as the stable outcome of an internal state-space process, rather than as an externally imposed choice rule. This provides a deterministic account of commitment formation, symmetry resolution through basins of attraction, and hysteresis and decision persistence under perturbations. For action spaces with circulant coupling, we prove using Lyapunov-Schmidt reduction that the action-coupling operator determines which components of evidence govern decision formation. We further show that the dynamics implicitly compute a geometry-aware utility, converge exponentially to the corresponding best response, and switch only when evidence is sufficiently strong. In contrast, supplying the same geometry-aware utility directly to logit dynamics does not recover these properties, showing that relational structure must be embedded in the decision mechanism itself. We illustrate the framework in a repeated coverage game, prove that the induced game is an exact potential game, and show that its Nash equilibria are reached by the neuromorphic decision dynamics.

\end{abstract}
\section{Introduction}
Decision-making problems often require an agent to select an action from a finite set and to do so repeatedly as conditions evolve over time~\cite{paccagnan2022utility}. When there are multiple agents and factors such as congestion limitations, shared resources, or coordination constraints, 
%In many such settings, 
actions are not isolated alternatives but rather carry relational structure.
% set carries relational structure beyond its labels: 
% nearby headings lead to similar motion, adjacent sectors overlap in coverage, and tasks sharing resources are partial substitutes. 
For example, when an agent selects a motion direction, nearby directions may yield similar motion outcomes. When choosing a sensing sector, adjacent sectors may overlap in coverage.  When allocating effort across candidate tasks, tasks requiring overlapping capabilities may be partial substitutes. When committing to one among several resource-seeking actions, actions close in heading or location may compete for the same resource.  Actions sharing downstream paths or bottlenecks may have coupled payoffs. 
When such relations matter, decision making should account not only for the utility of each action in isolation but also for how utility is distributed across related actions.

%\hs{the following explanation on action-coupling can be removed. the first paragraph, i think, motivates action relations}

%Action relations may admit positive coupling between similar or overlapping actions and negative coupling between competitive or incompatible actions. 
%a block structure, where each block represents a family of interchangeable or closely related choices.
% % such as neighboring sectors, geometrically similar motion commands, or tasks tied to the same resource class, Positive coupling within blocks captures similarity or overlap whereas negative coupling across blocks captures competition or incompatibility. Ring, lattice, hierarchical, and signed-partition structures are all special cases. In each instance, the 

%The coupling encodes how action-utility should spread, aggregate, or be contrasted across actions during decision formation.

Classic best response selects the action with highest utility under the current utility profile~\cite{swenson2018BR}. 
Smooth variants such as logit 
% replace this deterministic choice by 
use a probabilistic response map~\cite{mckelvey1995quantal,blume1993statistical}. These models prescribe an action-selection rule, but they do not leverage relational structure among actions beyond including action coupling in utility construction, which is not robust. 
% are useful at the level of response rules, 
They also fall short in dealing with time-varying, corrupted, noisy, and/or partial evidence~\cite{wang2023convergence}.  %and providing hysteresis and persistence. 
Importantly, classic approaches do not provide a mechanism for how best response is realized, how near-ties are resolved, or how commitment persists under transient perturbations. %Nor do they leverage relational structure among actions beyond including action coupling in utility construction, which is not robust. 

To address these shortcomings, we propose and study a mechanistic dynamical system formulation of best response as an internal decision process. Our formulation uses neuromorphic decision dynamics as a state-space realization of best response in finite-action games with relational structure on the action set. Each agent maintains an internal decision state over actions. External evidence enters as an input and a mixed-sign coupling operator on the action set shapes how evidence is integrated during decision formation. Action relations may admit positive coupling between similar or overlapping actions and negative coupling between competitive or incompatible actions. The selected action is realized by the stable attractor reached by the internal dynamics instead of an externally imposed $\arg\max$ or by stochastic sampling~\cite{auletta2011convergence}, as in logit-type dynamics that converge more slowly to a distribution over actions corresponding to a quantal equilibrium~\cite{auletta2011convergence}. Our decision dynamics go beyond classic best response by giving rich deterministic behavior including fast commitment through attractor structure, basin-dependent symmetry resolution, and, in the subcritical regime, hysteresis and persistence.

Our approach connects three lines of work. First, congestion and potential games provide the benchmark setting for repeated finite-action coordination and equilibrium convergence \cite{rosenthal1973class,monderer1996potential}. Second, smooth and stochastic response models study equilibrium selection through perturbed or probabilistic choice \cite{mckelvey1995quantal,blume1993statistical}. Third, the Nonlinear Opinion Dynamics (NOD) model shows that saturating recurrent interactions can generate multistability, tunable sensitivity to input, fast decision formation, and game-theoretic behaviors \cite{bizyaeva2023nonlinear,park2022tuning,moreno2024fast,amorim2024spatial,seelig2015neural}. \emph{Our contribution is to connect these lines through a mechanistic dynamical systems perspective, showing how neuromorphic decision dynamics, based on NOD, realize best response as the stable outcome of an internal state-space process and provide significant advantages to performance.}

% in structured action spaces: action-set coupling defines the relevant best-response notion, and neuromorphic dynamics realize it as a stable decision mechanism.

We specialize in this paper to symmetric circulant mixed-sign coupling among actions. This provides a canonical model of geometry-aware action spaces and admits an explicit spectral and bifurcation analysis. The dominant eigenspace of the coupling matrix determines which components of evidence govern equilibrium selection near onset, leaving the remaining noncritical components of input to be attenuated through stable modes. The result is an action-structure-aware best-response mechanism and, in the subcritical regime, a hysteretic one.
% \nl{Can you please list the contributions and then include a very brief paragraph that lists what is in each section. Here it is kind of confused what are the contributions and what is written in each section. You can look at Ian Xul's paper for reference.}
% \hs{yes, doing that right now.}
% \nl{and then please read through my edits to the introduction. I re-arranged a bit as I think it is important to lead with the notion that classic BR does not leverage action relations since action relations is how the paper starts and seems to be motivated.}
%We further prove that relational structure must be encoded in the decision mechanism itself by demonstrating  that providing geometry-aware utility directly to logit dynamics is not sufficient.
% \noindent\textbf{Contributions.} 
These are our \textbf{contributions}:
\begin{enumerate}
\item  We present a mechanistic neuromorphic formulation of best response decision dynamics in finite-action games with relational action structure using NOD. Our method realizes action-selection as the stable outcome of an internal state-space process.
\item We prove, using Lyapunov--Schmidt reduction, that a coupled action space governs the components of evidence that drive decision formation through the underlying coupling operator.
\item We demonstrate that the neuromorphic dynamics formulation implicitly computes geometry-aware utility, exponentially converges to the best response, 
% is independent of the number of actions,
and in the subcritical regime, committed decisions persist under perturbations and switch only with sufficient evidence.
\item We show that providing geometry-aware utility directly to logit dynamics is insufficient, emphasizing that the decision mechanism must encode relational structure underlying the action-set.
\item We instantiate our formulation in a repeated coverage game and prove that it is an exact potential game whose Nash equilibria are reached by neuromorphic dynamics.
\end{enumerate}
Section~\ref{sec:formulation} formulates the repeated game and the neuromorphic decision dynamics. Section~\ref{sec:bifurcation}
develops the bifurcation analysis and 
proves the realization of geometry-aware best response. Section~\ref{sec:game} studies the coverage game instantiation and proves convergence, optimality, and hysteresis.
% Section~\ref{sec:formulation} formulates repeated finite-action decision making over structured action spaces and introduces the neuromorphic decision dynamics. Section~\ref{sec:bifurcation} develops the bifurcation analysis for symmetric circulant mixed-sign couplings and shows how the dominant critical eigenspace governs equilibrium selection in both supercritical and subcritical regimes. Section~\ref{sec:game} studies a coverage-game instantiation and shows how the resulting stable attractors realize the corresponding geometry-aware best response, with hysteresis yielding persistent commitment under time-varying evidence. 
Section~\ref{sec:conclusion} presents conclusions.
% with a discussion on the implications of a mechanistic realization of best response in finite-action games.

%%%%%%%%%%%%%%%%%%%%%%%%%%%%%%%%%%%%%%%%%%%%%%%%%%%%%%%%%%%%%%%%%%%%%%%%%%%%%%%%%
\section{Problem Formulation}
\label{sec:formulation}
% \todo{boldsymbols for vectors}

\subsection{Problem Setup}
\label{sec:setup}
We consider a repeated multi-agent decision problem with agent set $\mathcal{N}=\{1,\dots,N\}$ and common finite action set $\mathcal{A}=\{1,\dots,K\}$. At each epoch, agent $i$ selects $a_i\in\mathcal{A}$. Let the joint action profile be $\va\in\mathcal{A}^N$ and let $\va_{-i}
% := (a_1,\dots,a_{i-1},a_{i+1},\dots,a_N) 
\in \mathcal{A}^{N-1}$ be the joint action profile of all agents except $i$. 
The environment provides each agent $i$ with an evidence vector $\vb_i(\va_{-i})\in\mathbb{R}^K$, where $b_{ik}(\va_{-i})$ is the raw signal to agent $i$ for action $k$, reflecting local observations about costs, rewards, or event densities, \nl{any of which may depend on  $\va_{-i}$
% , the joint profile of actions 
at the time of the observations}.
We equip $\mathcal{A}$ with a symmetric \emph{action-coupling operator} $A\in\mathbb{R}^{K\times K}$, where $A_{kj}$ encodes geometry or relational affinity between actions $k$ and $j$ (e.g., spatial proximity, functional substitutability, shared resources). We define \emph{utility} of agent $i$'s actions 
% is 
% then defined 
as 
the projection of its evidence vector onto the dominant eigenspace of $A$ \hs{and set}
$U_i(\va) = \Psi_A\!\left(\vb_i(\va_{-i})\right)$.
$\Psi_A:\mathbb{R}^K\to\mathbb{R}^K$ \hs{is defined as the orthogonal projection operator that} aggregates evidence across actions
weighted by coupling strength. 
$U_i \in \mathbb{R}^K$ 
% plays the role of 
is the effective utility in the game-theoretic sense~\cite{paccagnan2022utility}: it is the quantity whose maximization defines rational action selection when the action space carries relational structure.
The \emph{geometry-aware best response} for agent $i$ is then
\begin{align}
    \label{eq:projected_br}
    \mathrm{BR}^A_i(\va_{-i}) = \arg\max_{k \in \mathcal{A}}\, \left[\Psi_A\bigl(\vb_i(k,\, \va_{-i})\bigr)\right]_k,
\end{align}
A profile $\va^\star$ is a projected Nash equilibrium if $a^\star_i \in \mathrm{BR}^A_i(\va^\star_{-i})$ for all $i \in \mathcal{N}$. 
In the following, we propose neuromorphic decision dynamics that realize this best response.

% Classical best response is recovered when $\Psi^A = \mathrm{Id}$.
% This raises a natural question: does feeding $\Psi^A(b_i)$ into standard decision dynamics recover $\mathrm{BR}^A_i$? The answer is no. As we show in Section~\ref{sec:game}, logit dynamics given $\Psi^A(b_i)$ as utility fail to stabilize at $\mathrm{BR}^A_i$: geometry-aware decisions require $A$ to shape the decision dynamics directly
% as an action-coupling operator
% % , not only through evidence. 
% beyond its role in defining utility.
% We analyze the case where $A$ is circulant and show that neuromorphic decision dynamics provably select $\mathrm{BR}^A_i$ as a stable attractor.

\subsection{Neuromorphic Decision Dynamics}
\label{sec:dynamics}
We propose neuromorphic decision dynamics %(NDD) 
using NOD with a cyclic action set $\mathcal{A}$ equipped with a circulant, symmetric action-coupling operator $A$. The action set is viewed as a discrete circle, and $A$ defines a shift-invariant interaction pattern analogous to ring-attractor networks~\cite{amorim2024spatial} whose spectral structure we exploit in Section~\ref{sec:bifurcation}.
Each agent $i$ maintains an internal decision $\vz_i(t) \in \mathbb{R}^K$, where each component $z_i^k$ represents real-valued preference for action $k \in \mathcal{A}$ evolving according to NOD equations \cite{bizyaeva2023nonlinear,leonard2024fast}:
\begin{equation}
\label{eq:nod}
\tau_z \dot{\vz}_i(t) = -d_z \vz_i(t) + S\!\big(\alpha A \vz_i(t) + \vb_i\big),
\end{equation}
where $\tau_z > 0$, $d_z > 0$, $\alpha > 0$, and $\vb_i \in \mathbb{R}^K$ are the decision time constant, leakage rate, coupling gain, and evidence vector (Section~\ref{sec:setup}), respectively. The  nonlinearity $S : \mathbb{R} \to \mathbb{R}$ acts componentwise and is odd, $\mathcal{C}^5$, sigmoidal, with $S'(0) = 1$ and $s_3 \triangleq S'''(0) \neq 0$. Action selection corresponds to a stable equilibrium of  \eqref{eq:nod}, i.e., a solution of
\begin{equation}
\label{eq:F}
F(\vz, \alpha, \vb) \triangleq -d_z \vz + S(\alpha A \vz + \vb) = \vzero.
\end{equation}
Gain $\alpha$ acts as the %primary 
bifurcation parameter: as $\alpha$ increases, the trivial equilibrium $\vz = \vzero$ loses stability and committed decisions emerge as stable nontrivial attractors~\cite{bizyaeva2023nonlinear,leonard2024fast}.

\noindent\emph{Circulant Action Coupling:}
We index $\mathcal{A}$ by equally spaced points on a discrete circle with sector angles $\theta_j = 2\pi j/K$, and take $A$ to be real, symmetric, and circulant: $A_{jk} = \varphi(j-k \bmod K)$ for a symmetric profile $\varphi$. The canonical choice is a Mexican-hat profile with local excitation and surround inhibition, for example, \texttt{toeplitz}$\left([+, +, 0, -, \ldots, -, 0, +]\right)$. Following~\cite{graytoeplitz}, %By standard results
$A$ is diagonalized by the real discrete Fourier basis:
$\vphi_k \triangleq [\cos(k\theta_j)]_{j=0}^{K-1}$,
$\vpsi_k \triangleq [\sin(k\theta_j)]_{j=0}^{K-1}$,
$k = 0,\ldots,\lfloor K/2\rfloor$,
with $A\vphi_k = \lambda_k\vphi_k$, $A\vpsi_k = \lambda_k\vpsi_k$, real
eigenvalues satisfying $\lambda_k = \lambda_{K-k}$, and orthogonality 
$\langle\vphi_k,\vphi_k\rangle = \langle\vpsi_k,\vpsi_k\rangle = \tfrac{1}{2}$,
$\langle\vphi_k,\vpsi_k\rangle = 0$ under $\langle \boldsymbol{u},\boldsymbol{v}\rangle \triangleq
\tfrac{1}{K}\sum_j u_j v_j$.
The linearization of~\eqref{eq:F} at $\vz=0$ is $L(\alpha) = -d_z I + \alpha A$, sharing eigenvectors with $A$ and scalar gain $-d_z + \alpha\lambda_k$ on eigenmode $k$. Thus, the trivial equilibrium loses stability eigenmode-by-eigenmode as $\alpha$ increases, first along the eigenmode with the largest eigenvalue.

\begin{assumption}
% [Unique dominant mode]
\label{ass:dominant_mode}
There exists a unique index $k_\star \in \{1,\dots,\lfloor\frac{K-1}{2}\rfloor\}$ such that $\lambda_{k_\star} > \lambda_k$ for all \hs{$k\in\{0,\ldots,\lfloor K/2\rfloor\}\setminus\{k_\star\}$}.
\end{assumption}

Under Assumption~\ref{ass:dominant_mode}, the critical gain is $\alpha_c = d_z / \lambda_{k_\star}$, and instability first occurs  along the two-dimensional subspace $E_c = \mathrm{span}\{\vphi_{k_\star}, \vpsi_{k_\star}\}$ where $\mu = -d_z + \alpha\lambda_{k_\star}$ measures distance from bifurcation. 
% \begin{remark}
%     For a Mexican-hat coupling profile, $k_\star = 1$, so the dominant unstable mode is the first Fourier harmonic. All remaining modes span the stable complement $E_s = E_c^\perp$ and remain strictly damped at $\alpha_c$.
% \end{remark}
{\begin{assumption}[Local near-bifurcation regime]
\label{ass:local_regime}
% For bifurcation analysis, 
We focus on a neighborhood of the critical gain $\alpha_c = \frac{d_z}{\lambda_{k_\star}}$.
assuming that $|\mu|, \|\vb\|$, and $\|\vz\|$ are sufficiently small,
equivalently, that $|\alpha-\alpha_c|$ is small and the equilibrium under consideration lies in a sufficiently small neighborhood of $\vz=\vzero$.
\end{assumption}}

\section{Analysis of Neuromorphic Decision Dynamics}
\label{sec:bifurcation}
% \todo{need to write explicitly that $\vb$ is small, $\vz$ is small }
% The local equilibrium structure of ~\eqref{eq:F} changes
% qualitatively 
% \vspace{-1em}
As $\alpha$ increases through $\alpha_c$, $\vz = \vzero$ loses stability, and a family of committed decision states emerges (Figure~\ref{fig:bd}). Under Assumptions~\ref{ass:dominant_mode} and~\ref{ass:local_regime}, 
we formalize this transition by reducing the full $K$-dimensional equilibrium equation to a two-dimensional problem on $E_c$ using Lyapunov--Schmidt reduction~\cite{golubitsky2012singularities}.
% \vspace{-2em}

\subsection{Low-dimensional reduction}
\label{subsec:LS}
At $\alpha = \alpha_c$, the linearization $L(\alpha_c)$ is singular: its kernel is $E_c$ and, since $L(\alpha_c)$ is self-adjoint with respect to $\langle\cdot,\cdot\rangle$, its range is $E_s = E_c^\perp$. Standard implicit function theorem arguments therefore fail at this point, and a reduction is required to characterize equilibria near the onset of decision formation. Write $\vz = \vv + \vw$ with $\vv \in E_c$ and $\vw \in E_s$. Projecting \eqref{eq:F} onto $E_s$ and $E_c$ yields the coupled system 
$\mathcal{G}(\vv,\vw,\alpha,\vb):= \Psi_A F(\vv + \vw, \alpha, \vb)$ and $\mathcal{H}(\vv,\vw,\alpha,\vb):=\Psi_A^\perp F(\vv + \vw, \alpha, \vb)$.
% \begin{equation}
% \label{eq:LS_split}
% \underbrace{\Psi_A F(\vv + \vw, \alpha, \vb)}_{\mathcal{G}(\vv,\vw,\alpha,\vb)} = 0,
% \qquad
% \underbrace{\Psi_A^\perp F(\vv + \vw, \alpha, \vb)}_{\mathcal{H}(\vv,\vw,\alpha,\vb)} = 0
% \end{equation}
Since all noncritical eigenvalues of $L(\alpha_c)$ are strictly negative, $D_w \mathcal{H}(\vzero,\vzero,\alpha_c,\vzero) = \Psi_A^\perp L(\alpha_c)$ is invertible. By the implicit function theorem, there exists a unique $\mathcal{C}^3$ \hs{local} reduction map $W : E_c \times \mathbb{R} \times \mathbb{R}^K \to E_s$ solving $\mathcal{H} = \vzero$ identically, with $W(\vzero,\alpha_c,\vzero) = \vzero$ and $D_v W(\vzero,\alpha_c,\vzero) = \vzero$. The latter follows because $L(\alpha_c)\delta \vv \in E_c$ for any $\delta \vv \in E_c$, which is annihilated by $\Psi_A^\perp$, so variations in $\vv$ do not linearly force the range equation. The odd symmetry of $S$ eliminates quadratic terms, giving $W(\vv,\alpha,\vb) = \mathcal{O}(\|\vb\|) + \mathcal{O}(\|\vv\|^3)$: the stable modes are uniquely determined (locally) to the cubic and linear-in-evidence corrections of the critical mode. Substituting $\vw = W(\vv,\alpha,\vb)$ into the critical equation yields the reduced condition $\mathcal{G}(\vv,\alpha,\vb) = \vzero$ posed entirely on $E_c$.
Let $\circ$ denote Hadamard  (entry-wise) product.
\begin{lemma}[Reduced critical equation]
\label{lem:G_reduced}
{Under Assumptions~\ref{ass:dominant_mode} and~\ref{ass:local_regime}},
%For $(\mu, b, v)$ sufficiently small,
\begin{align}
\label{eq:G_expanded}
\mathcal{G}(\vv,\mu,\vb)
&= \mu \vv + \Psi_A(\vb)
+ \frac{s_3}{6}\alpha_c^3\,\Psi_A\!\big((A\vv)^{\circ 3}\big) \nonumber \\
&\quad + \frac{s_3}{2}\alpha_c^2\,\Psi_A\!\big((A\vv)^{\circ 2} \circ \vb\big)
+ \mathcal{R},
\end{align}
where
$
\mathcal{R}
=
\mathcal{O}(\|\vv\|^5)
+\mathcal{O}(|\mu|\|\vv\|^3)
+\mathcal{O}(\|\vb\|\|\vv\|^2)
+\mathcal{O}(\|\vb\|^2\|\vv\|)
+\mathcal{O}(\|\vb\|^3)
$. 
Here $\mu\vv$ controls the onset of instability, $\Psi_A(\vb)$ is the leading projected evidence, the cubic self-interaction saturates amplitude, and the mixed cubic term gives the leading amplitude-dependent input correction.
% The linear term $\mu \vv$ controls onset of instability, the cubic self-interaction saturates decision amplitude, and the mixed cubic term captures how evidence $\vb$ biases the emerging decision through the coupling structure $A$.
\end{lemma}

\begin{figure}[htpb]
    \centering
    \includegraphics[width=0.9\linewidth]{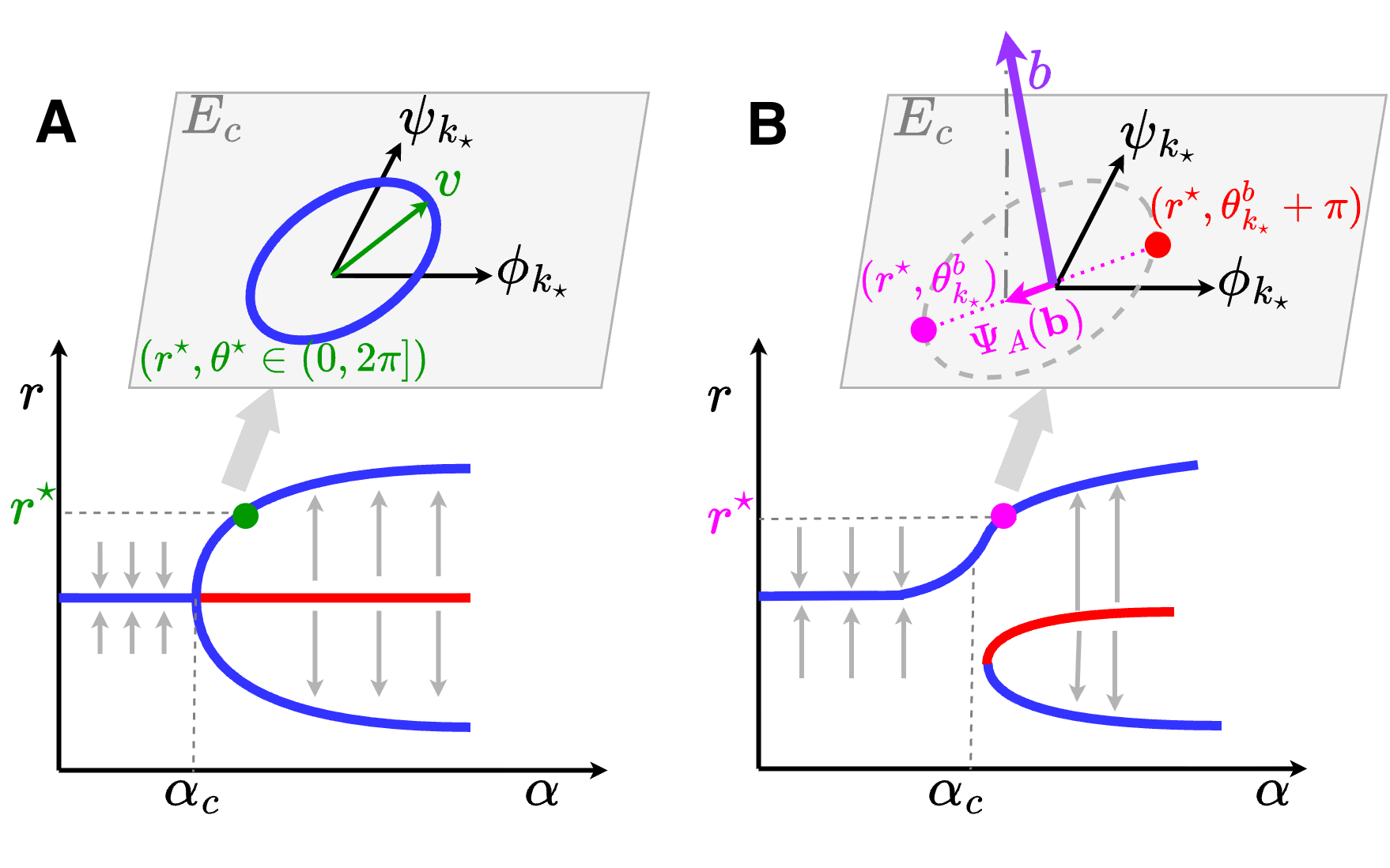}
    \caption{
    \textbf{Bifurcation geometry and action-selection in the critical eigenspace $E_c = \mathrm{span}\{\vphi_{k_\star}, \vpsi_{k_\star}\}$.}
    \textbf{(A)} At the onset $\alpha=\alpha_c$, a supercritical pitchfork generates a continuum of equilibria parameterized by $(r^\star,\theta^\star)$, forming a ring in $E_c$ (blue) of phase-indeterminate committed states along the stable branch $r=r^\star$.
    \textbf{(B)} In the presence of an input $\vb$ (purple), its projection $\Psi_A(\vb)$ onto $E_c$ (magenta) selects a unique equilibrium on the ring. The intersection determines the phase $\theta^b_{k_\star}$ and amplitude $r^\star$, yielding the committed state $(r^\star,\theta^b_{k_\star})$. 
    }
    \label{fig:bd}
    \vspace{-1.5em}
\end{figure}

\begin{proof}
Expand $F$ in Eq.~\eqref{eq:F} using $S(u) = u + \tfrac{s_3}{6}u^{\circ 3} + \mathcal{O}(\|u\|^5)$ with $\mathbf{u}=\alpha A(\mathbf{v}+\mathbf{w})+\mathbf{b}$, and project onto $E_c$:
% \begin{equation*}
$\mathcal{G} = \Psi_A L(\alpha)(\vv+\vw) + \Psi_A(\vb)
+ \tfrac{s_3}{6} \Psi_A\!\big(\alpha A(\vv+\vw)+\vb\big)^{\circ 3}
+ \mathcal{O}_{\geq 5}$.
% \end{equation*}
For the linear terms: $\vv \in E_c$ is an eigenvector of $A$ with eigenvalue $\lambda_{k_\star}$, so $\Psi_A L(\alpha)\vv = \mu \vv$. Since $E_s$ is invariant under $A$ for any symmetric circulant matrix~\cite{graytoeplitz} (each real Fourier plane is an eigenspace), $A\vw \in E_s$ and thus $\Psi_A L(\alpha)\vw = \vzero$. The linear contribution is exactly $\mu \vv + \Psi_A(\vb)$.
For the cubic term, since $\alpha=\alpha_c+\mu/\lambda_{k_\star}$, replacing $\alpha$ by $\alpha_c$ in the retained cubic coefficients introduces the additional remainder $\mathcal{O}(|\mu|\|\vv\|^3)$, 
while \hs{we collect higher-order terms and contributions through stable modes in $\mathcal R$, including mixed cubic }corrections involving $\vb$ into $\mathcal{O}(\|\vb\|\|\vv\|^2)$ \hs{which are  attenuated through the stable modes}.
Thus, retaining the leading cubic contributions,
$
\big(\alpha A(\vv+\vw)+\vb\big)^{\circ 3}
=
(\alpha_c A\vv)^{\circ 3}
+3(\alpha_c A\vv)^{\circ 2}\circ\vb
+\mathrm{rem.}
$
Using $\|\vw\| = \mathcal{O}(\|\vv\|^3)+\mathcal{O}(\|\vb\|)$, the remaining terms satisfy
$(\alpha_c A\vv)^{\circ 2}\circ A\vw
=
\mathcal{O}(\|\vv\|^5)
+
\mathcal{O}(\|\vb\|\|\vv\|^2)
$,
$(\alpha_c A\vv)\circ(\alpha_c A\vw+\vb)^{\circ 2}
=
\mathcal{O}(\|\vb\|^2\|\vv\|)+\mathrm{h.o.t.}$,
and
$(\alpha_c A\vw+\vb)^{\circ 3}
=
\mathcal{O}(\|\vb\|^3)+\mathrm{h.o.t.}$.
The projected cubic expansion contains the displayed terms
% The $A$-invariance of $E_s$ ensures $\Psi_A(A\vw) = 0$, so the cubic projection reduces to 
$\Psi_A((A\vv)^{\circ 3})$ and $\Psi_A((A\vv)^{\circ 2} \circ \vb)$,
while the remaining contributions %, including those involving $\vw$, 
are absorbed into $\mathcal{R}$  
% Collecting terms 
yielding~\eqref{eq:G_expanded}.
% without cross terms between $E_c$ and $E_s$. Collecting retained terms and bundling higher-order contributions into $\mathcal{R}_{\geq 4}$ yields~\eqref{eq:G_expanded}.
\end{proof}

\subsection{Polar Reduced Equations}
\label{subsec:polar}

Any $\vv \in E_c$ admits a polar representation $\vv = r e_\theta$ with $r \geq 0$ and $e_\theta = \sqrt{2}(\cos\theta\, \vphi_{k_\star} + \sin\theta\, \vpsi_{k_\star})$. The orthonormal frame $\{e_\theta, e_{\theta+\pi/2}\}$ on $E_c$ yields two scalar equations by projecting \eqref{eq:G_expanded} onto the radial and angular directions. The evidence $\vb \in \mathbb{R}^K$ decomposes over the Fourier basis as $b_j = a_0 + \sum_m (a_m \cos(m\theta_j) + \beta_m \sin(m\theta_j))$, with amplitude $A_m^b = \sqrt{a_m^2 + \beta_m^2}$ and phase $\theta_m^b = \operatorname{atan2}(\beta_m, a_m)$ at mode $m$, 
for $(a_m, \beta_m)\ne (0,0)$.

\begin{proposition}[Local polar reduced equations]
\label{prop:polar}
With $\Gamma = \frac{s_3}{4}\alpha_c^3\lambda_{k_\star}^3$ and 
$\eta = \frac{s_3}{4}\alpha_c^2\lambda_{k_\star}^2$, the reduced 
equilibrium conditions $g_r = 0$ and $g_\theta = 0$ are, up to cubic order,
%For $b = 0$, $g_{\theta} = 0$ and,
\begin{align}
g_r &= \mu r + \Gamma r^3
+ \frac{A_{k_\star}^b}{\sqrt{2}}\cos(\theta - \theta_{k_\star}^b)
\label{eq:polar_r} \\
&\quad + \frac{\eta r^2}{\sqrt{2}}\Big[
3A_{k_\star}^b\cos(\theta - \theta_{k_\star}^b)
+ A_{3k_\star}^b\cos(3\theta - \theta_{3k_\star}^b)
\Big] \nonumber\\
% \end{align}
% \begin{align}
%&\text{For } b \neq 0,\quad 
g_\theta &= \boldsymbol{1}(b\ne 0)\bigg(\frac{A_{k_\star}^b}{\sqrt{2}}\sin(\theta_{k_\star}^b - \theta)
\label{eq:polar_theta} \\
&\quad + \frac{\eta r^2}{\sqrt{2}}\Big[
3A_{k_\star}^b\sin(\theta_{k_\star}^b - \theta)
+ A_{3k_\star}^b\sin(\theta_{3k_\star}^b - 3\theta)
\Big]\bigg). \nonumber
\end{align}
\end{proposition}
\vspace{0.5em}

\begin{proof}
Substitute $v = r\ve_\theta$ into the reduced equation \eqref{eq:G_expanded} from Lemma~\ref{lem:G_reduced}. Since $\ve_\theta$ is an eigenvector of $A$ with eigenvalue $\lambda_{k_\star}$, we have $A\vv = \lambda_{k_\star} r \ve_\theta$, and thus $(A\vv)^{\circ 3} = \lambda_{k_\star}^3 r^3\, \ve_\theta^{\circ 3}$ and $(A\vv)^{\circ 2} \circ \vb = \lambda_{k_\star}^2 r^2\, \ve_\theta^{\circ 2} \circ \vb$.
For the cubic self-interaction, the identity $\Psi_A(\ve_\theta^{\circ 3}) = \tfrac{3}{2\sqrt{2}} \ve_\theta$  holds for circulant $A$~\cite{graytoeplitz}, yielding the coefficient $\Gamma = \frac{s_3}{6}\alpha_c^3\lambda_{k_\star}^3 \cdot \frac{3}{2} = \frac{s_3}{4}\alpha_c^3\lambda_{k_\star}^3$.
For the mixed cubic term, decompose $\ve_\theta^{\circ 2} \circ \vb$ in the Fourier basis. Using $\ve_\theta^{\circ 2} = \frac{1}{\sqrt{2}} (\Phi_0 + \cos 2\theta\,\Phi_{2k_\star} + \sin 2\theta\,\Psi_{2k_\star})$, the convolution with $\vb$ projects onto modes $k_\star$ and $3k_\star$ 
under $\Psi_A$: 
$
\Psi_A(\ve_\theta^{\circ 2} \circ \vb) = \frac{1}{\sqrt{2}}\Big[
    A_{k_\star}^b \cos(\theta - \theta_{k_\star}^b)\,\ve_\theta
    + A_{3k_\star}^b \cos(3\theta - \theta_{3k_\star}^b)\,\ve_\theta
    + \text{angular terms}
    \Big]$,
    with coefficient $\eta = \frac{s_3}{4}\alpha_c^2\lambda_{k_\star}^2
$.
Similarly, projecting $\mathcal{G} = \vzero$ onto $\ve_\theta$ (radial) and 
$\ve_{\theta+\pi/2}$ (angular) yields \eqref{eq:polar_r}--\eqref{eq:polar_theta} directly, while the remaining contributions, including those induced by the reduction map $\vw$ are absorbed into the stated remainder terms in~\eqref{eq:G_expanded}.
\end{proof}

\subsection{Supercritical Bifurcation and Decision Formation}
\label{subsec:supercritical}
\hs{We characterize the unforced committed states,  establish the role of noncritical evidence and then analyze how forcing determines the committed phase and corresponding action.}
% \vs{Below: we should show exactly what is $g_\theta$ when $b=0$ since $\theta^b$ is not defined. Also, while its clear from structure, it should be shown with a Lyapunov function or other technique that $(r^*, \theta^b)$ is stable. } \hs{done}

%\paragraph{Unforced case ($\vb = 0$)} 
\emph{Unforced case ($\vb=\vzero$):} The polar equations reduce to $g_r=\mu r+\Gamma r^3$ and $g_\theta\equiv 0$. The radial normal form has nonzero solutions $r=\pm\sqrt{-\mu/\Gamma}$ for $\Gamma<0$, $\mu>0$. In polar coordinates, $r\ge 0$, so the negative solution is interpreted as a phase shift by $\pi$: $-\sqrt{-\mu/\Gamma}\,\ve_\theta =\sqrt{-\mu/\Gamma}\,\ve_{\theta+\pi}$. Thus the signed branches form the committed ring $\{r^\star\ve_\theta:\theta\in[0,2\pi)\}$ with $r^\star=\sqrt{-\mu/\Gamma}$, as shown in Figure~\ref{fig:bd}. The equilibrium direction $\theta^\star$ is indeterminate and selected by initial conditions. At $(r^\star,\theta^\star)$, $\partial_r g_r(r^\star)=\mu+3\Gamma(r^\star)^2=-2\mu<0$, while $\partial_\theta g_\theta(r^\star,\theta^\star)=0$.
% , reflecting amplitude stability and neutral phase degeneracy.
Thus the branch is locally stable in
amplitude, while the 0 angular eigenvalue reflects neutral
phase degeneracy due to rotational symmetry.

\hs{\textit{Forced case ($\vb\neq\vzero$):}
We next determine how evidence resolves the unforced phase indeterminacy.
% , we first characterize the response of the noncritical modes.
}
\begin{lemma}[Suppression of noncritical evidence]
\label{lem:suppression}
Let $\vz=\vv+\vw$ with $\vv\in E_c$ and $\vw\in E_s$ be the Lyapunov--Schmidt decomposition. For each noncritical Fourier mode $m\neq k_\star$, the corresponding coefficient of $\vw$ satisfies $w_m=-\frac{b_m}{\sigma_m}+\mathrm{h.o.t.}$, where $\sigma_m:=-d_z+\alpha_c\lambda_m<0$ is the stable eigenvalue of $L(\alpha_c)$ at mode $m$. Hence, if
\begin{equation}
\label{eq:suppression}
|\sigma_m|\gg A_m^b, \qquad \forall\, m\neq k_\star,
\end{equation}
noncritical harmonics enter the reduced critical equation only through higher-order corrections in $\vw$ and do not affect phase selection at leading order.
\end{lemma}
\begin{proof}
The range equation is $\mathcal H(\vv,\vw,\alpha,\vb)=\Psi_A^\perp F(\vv+\vw,\alpha,\vb)=\vzero$. Since $D_w\mathcal H(0,0,\alpha_c,0)=\Psi_A^\perp L(\alpha_c)$ is invertible on $E_s$, the implicit function theorem gives $\vw=-\bigl(L(\alpha_c)\bigr)^{-1}\!\left(\Psi_A^\perp \vb+\mathrm{h.o.t.}\right)$. In the Fourier basis of the circulant operator $A$, this yields $w_m=-b_m/\sigma_m+\mathrm{h.o.t.}$ for each $m\neq k_\star$. Therefore, under \eqref{eq:suppression}, these modes are attenuated by the stable spectrum and contribute only higher-order corrections to the reduced angular equation.
\end{proof}
%paragraph{Forced case ($\vb \neq 0$).}
% \emph{Forced case ($\vb \neq \vzero$):}

Evidence breaks rotational symmetry and
by Proposition~\ref{prop:polar},
$g_\theta=
\frac{A_{k_\star}^b}{\sqrt{2}}\sin(\theta_{k_\star}^b-\theta)
+ 
\frac{\eta r^2}{\sqrt{2}}
\Big[
3A_{k_\star}^b\sin(\theta_{k_\star}^b-\theta)
+
A_{3k_\star}^b\sin(\theta_{3k_\star}^b-3\theta)
\Big]$.
By Lemma~\ref{lem:suppression}, harmonics $m\neq k_\star$ are filtered through the stable component $w$ and do not enter the phase equation at leading order. The only leading correction on $E_c$ is the $3k_\star$ harmonic generated by cubic mixing. If $3A_{k_\star}^b>A_{3k_\star}^b$, the critical harmonic dominates near onset, so $g_\theta=0$ has, to leading order, the two solutions $\theta^\star=\theta_{k_\star}^b$
and $\theta^\star=\theta_{k_\star}^b+\pi$.
At these two phases, the Jacobian of the reduced equilibrium map satisfies 
$\partial_\theta g_\theta(\theta_{k_\star}^b)
=
-\frac{A_{k_\star}^b}{\sqrt{2}}+\mathcal O(r^2)<0$,
and
$\partial_\theta g_\theta(\theta_{k_\star}^b+\pi)
=
\frac{A_{k_\star}^b}{\sqrt{2}}+\mathcal O(r^2)>0,
$
for sufficiently small $r$. Thus the aligned phase is locally stable, the anti-aligned phase is unstable, and forcing pins the selected phase to the phase of the projected evidence $\Psi_A(\vb)$ as shown in Figure~\ref{fig:bd}.

\subsection{Subcritical Bifurcation and Hysteretic Best Response} \label{subsec:subcritical} When $\Gamma>0$, the cubic term destabilizes the bifurcating branch. To obtain a bounded subcritical bifurcation, we introduce $\alpha=\alpha_0+\kappa\|\vz\|^2$, giving $\mu=\mu_0+\kappa\lambda_{k_\star}r^2$, where $\mu_0=-d_z+\alpha_0\lambda_{k_\star}$ and $\|\vz\|^2=r^2+\mathcal O(r^6)$. The radial reduced equation becomes $ g_r=\mu_0r+\tilde\Gamma r^3+\Delta r^5 +[\Psi_A(\vb)]^\top\ve_\theta=0, $ where $\tilde\Gamma=\Gamma+\kappa\lambda_{k_\star}$, $\Delta<0$, and $[\Psi_A(\vb)]^\top\ve_\theta =\frac{A^b_{k_\star}}{\sqrt{2}}\cos(\theta-\theta^b_{k_\star})$ is the projected evidence along the current phase. Subcriticality requires $\tilde\Gamma>0$, equivalently $\kappa>-\Gamma/\lambda_{k_\star}$. As in the supercritical case, the radial normal form may be read in signed amplitude along a fixed direction $\ve_\theta$. Let $\mu_{\mathrm{SN}}<0$ denote the saddle-node point of the unforced radial equation, defined by $g_r=0$ and $\partial_r g_r=0$. For $\mu_0\in(\mu_{\mathrm{SN}},0)$, the stabilized subcritical pitchfork has coexisting neutral and finite-amplitude committed attractors, separated by unstable branches. The signed committed states $r=\pm r^\star$ correspond in polar coordinates to phases separated by $\pi$, since $-r^\star\ve_\theta=r^\star\ve_{\theta+\pi}$. 
% Thus the unstable branches define thresholds for the onset or loss of finite-amplitude commitment. 
For decision making, $\mu_0$ is fixed in this coexistence regime, while projected evidence unfolds the radial equation relative to the current phase. 
% A weak destabilizing radial component $[\Psi_A(\vb)]^\top\ve_\theta$ perturbs the committed amplitude but does not erase commitment. A sufficiently strong component removes the finite-amplitude attractor through a saddle-node and induces a jump. 
% The transverse component enters $g_\theta$ and can reorient the committed phase. 
% Hence the subcritical radial structure gives hysteresis in commitment amplitude, while $g_\theta$ determines whether the selected action persists or reallocates. 
% The selected action is unchanged when the finite-amplitude state survives and the phase remains in the same action-selection sector; otherwise, the phase realigns with the dominant projected input.
% as in Section~\ref{subsec:supercritical}.
\hs{
\begin{proposition}[Hysteresis in decision commitment]
\label{prop:hysteresis}
Assume the state-dependent-gain extension of Section~\ref{subsec:subcritical} with $\tilde\Gamma>0$, $\Delta<0$, and $\mu_0\in(\mu_{\mathrm{SN}},0)$. Let $(r^\star,\theta_0)$ be a locally stable committed equilibrium. 
For the radial unfolding at phase $\theta_0$, there exists $\delta^\star>0$ such that the committed branch remains stable whenever
$
|[\Psi_A(\delta\mathbf{b}_i)]^\top \mathbf{e}_{\theta_0}|<\delta^\star
$.
% in the destabilizing direction, 
% transient fluctuations $\delta\mathbf{b}_i(t)$ do not destroy finite-amplitude commitment. 
Quasistatic passage through this threshold in the destabilizing direction destroys the branch through a saddle-node and induces a finite-amplitude transition. The selected action persists while commitment is maintained and the phase remains in the same action-selection sector; otherwise, the phase evolution in~\eqref{eq:polar_theta} determines the selected action.
\end{proposition}

\begin{proof}
% By Section~\ref{subsec:subcritical}, 
% For $\mu_0\in(\mu_{\mathrm{SN}},0)$ the committed branch is separated from the neutral state by an unstable branch, and projected input along $\mathbf{e}_{\theta_0}$ unfolds the radial equation. Hence a nonzero saddle-node threshold $\delta^\star$ exists: below it the finite-amplitude committed branch persists, while above it the branch is lost and the state jumps. 
% The second conclusion follows from Equation~\eqref{eq:polar_theta}.
% % selected action persists if this radial commitment is maintained and the phase displacement induced by $g_\theta$ stays within the same action-selection sector; otherwise, the angular dynamics determine the reallocated action from the projected input.
By Section~\ref{subsec:subcritical}, projected evidence along $\mathbf{e}_{\theta_0}$ unfolds the radial equation. The stable committed branch persists until it coalesces with the intervening unstable branch at a saddle-node satisfying $g_r=0$ and $\partial_r g_r=0$. This yields a nonzero threshold $\delta^\star$ and a finite-amplitude transition upon quasistatic passage through the fold. The action-persistence statement follows from
% the action-selection sectors and 
the phase equation~\eqref{eq:polar_theta}.
\end{proof}
}

\subsection{Realization of Best Response}
\label{subsec:proj_BR}
The bifurcation analysis reveals that only the projection of $\vb$ onto $E_c$ enters the reduced equations at leading order, $\Psi_A(\vb) = A_{k_\star}^b \big(\cos\theta_{k_\star}^b\,\vphi_{k_\star} + \sin\theta_{k_\star}^b\,\vpsi_{k_\star}\big)$,
which is the utility $U_i = \Psi_A(\vb_i)$ of Section~\ref{sec:formulation}: a weighted sum of evidence along the directions $\{\vphi_{k_\star}, \vpsi_{k_\star}\}$, 
%with amplitude $A_{k_\star}^b$ and phase $\theta_{k_\star}^b$ 
fully characterizing how evidence biases the emergent decision. Since $\theta^\star = \theta_{k_\star}^b$, the selected action $a_i^\star = \arg\max_k U_{ik}$ is $\mathrm{BR}_i^A$ of~\eqref{eq:projected_br}. 
{Hence, under Assumption~\ref{ass:local_regime}, and provided
\(3A^b_{k_\star}>A^b_{3k_\star}\) and \eqref{eq:suppression} holds, the
action selected by the locally stable equilibrium of \eqref{eq:polar_theta}
coincides, near onset, with the projected best response \(\mathrm{BR}_i^A\)
\eqref{eq:projected_br}.}
% \begin{equation}
%     \Psi_A(\vb) = A_{k_\star}^b \big(\cos\theta_{k_\star}^b\,\phi_{k_\star}
%     + \sin\theta_{k_\star}^b\,\psi_{k_\star}\big),
% \end{equation}
% By the reduced phase equation~\eqref{eq:polar_theta}, and under the conditions
% $
% 3A_{k_\star}^b>A_{3k_\star}^b
% $
% and \eqref{eq:suppression}, the locally stable equilibrium satisfies
% $
% \theta^\star=\theta_{k_\star}^b
% $
% near onset. Hence the selected action
% $
% a^\star=\arg\max_k U_{ik}
% $
% coincides locally with the projected best response $\mathrm{BR}_i^A$ of \eqref{eq:projected_br}.
For $k_\star=1$, $\Psi_A(\vb_i)$ is unimodal, so the induced choice is winner-take-all. 
{This perspective also suggests a design interpretation of the evidence projection $\Psi_A$. The raw evidence $\vb_i \in \mathbb{R}^K$ may contain noise or sensing features that should not directly determine the action. The projection $\Psi_A$ specifies which components of the input are decision-relevant and which are attenuated by stable modes. Thus, designing $\Psi_A$ amounts to choosing what information the best response should depend on. 
% For example, in coverage problems, raw evidence may encode sector-wise measurements or event density, while the utility used for action selection is a projected evidence field that emphasizes task-relevant spatial structure.
% For example, in symmetric settings (ring), the admissible eigenvectors are fixed by symmetry, and design acts through eigenvalue placement among Fourier modes.
This leads to an inverse design problem: given a desired decisive subspace $E_d$ of dimension $m < K$, synthesize an action-coupling operator $A$ whose dominant eigenspace is $E_d$. With graph or locality constraints, this becomes a structured eigenstructure assignment problem. The circulant Mexican-hat case studied here is one tractable example, where symmetry fixes the Fourier eigenspaces and design reduces to eigenvalue placement. More generally, $A$ need not be circulant; different task geometries or utility definitions can be encoded through the dominant eigenspace, spectral gaps, and sparsity pattern of $A$.}
% {This perspective also suggests a interaction-matrix design problem: given a desired projection of evidence onto a low-dimensional subspace, synthesize an action-coupling operator $A$ whose dominant eigenstructure realizes that projection. With graph constraints, this becomes a structured eigenstructure assignment problem.
% % ; with additional Mexican-hat coupling restriction, it reduces to selecting coupling parameters so that the desired structured mode is dominant. 
% In symmetric settings (ring), the admissible eigenvectors are fixed by symmetry, and design acts through eigenvalue placement among Fourier modes.
% % Finally,
% While the evidence lies in $\mathbb{R}^K$, the effective feature dimension is determined by the dominant eigenspace of $A$. }
In this sense, the proposed dynamics combine linear function approximation with a mechanistic dynamical-systems realization of best response.

\section{A Coverage Game Analysis}
\label{sec:game}
% \todo{include discussion suggested by Vaibhav here.}
{In the proposed framework, the projected utility $U_{ik}=[\Psi_A(\vb_i(t))]_k$ is a linear projection of the evidence field, with the projection determined by the dominant eigenspace of the action-coupling operator $A$. Thus, $A$ specifies which evidence components become decision-relevant in realizing the best response. The coverage game is a natural example because the action utility is a linear functional of the underlying evidence. The same structure also arises in other finite-action problems with relational action structure, e.g., in correlated task allocation, discretized heading selection, and structured hypothesis selection.}

Consider the coverage game with a discrete ring of $K$ sectors $\mathcal{A}=\{1,\dots,K\}$ and event density $V:\mathcal{A} \times \mathbb{R}_{\ge0}\to\mathbb{R}_{\geq 0}$, $\sum_k V(k,t )=1$, which varies on a slow time scale $\tau_b\gg\tau_z$ relative to~\eqref{eq:nod}. Each agent observes evidence $b_{ik}(t) = V(k,t) - \rho\sum_{j\neq i}\mathbf{1}[a_j(t)=k]$, $\rho>0$, balancing event density against congestion~\cite{marden2009cooperative}. For any symmetric circulant $A$ satisfying Assumption~\ref{ass:dominant_mode}, 
the local bifurcation analysis of Section~\ref{sec:bifurcation} 
implies, under Assumption~\ref{ass:local_regime}, that the
NOD~\eqref{eq:nod} implicitly computes the effective utility
$U_{ik}(\va_{-i}, t)= [\Psi_A(\vb_i(t))]_k
      = \sum_{s\in\mathcal A} C_{ks}\, b_{is}(t)$,
where $C_{ks} := \frac{2}{K}\cos\!\big(k_\star(\theta_s-\theta_k)\big)$ 
are the entries of the symmetric interaction kernel $C$ induced by the dominant mode, and the locally stable equilibrium of~\eqref{eq:nod}
%of NDD 
selects the corresponding projected best response near bifurcation onset. 
%  $ U_{ik}(t) = [\Psi_A(b_i(t))]_k
%   = \sum_{s\in\mathcal{A}} b_{is}(t)\cos\!\big(k_\star(\theta_s-\theta_k)\big)$,
% the {first moment of the evidence with respect to the cosine distance from sector $k$.} 
These utilities are the marginal increments of the welfare function~\eqref{eq:potential}, so the game is an exact potential game (Lemma~\ref{lem:potential}) {whose maximizers are Nash equilibria reached exponentially fast at rate $\text{min}(1,2|\mu|)$}, independent of $K$ (Proposition~\ref{prop:NE}).
% , and globally optimal in the $k_\star$-dominant regime (Corollary~\ref{cor:optimal}). 
Under subcritical bifurcation, amplitude and phase hysteresis ensures 
% that the committed decision state is 
robustness to perturbations below $\delta^\star$ (Proposition~\ref{prop:hysteresis}).

\begin{lemma}[Exact potential for the projected game]
\label{lem:potential}
Fix \(t\). Define $\bar V_k(t) := \sum_{s\in\mathcal A} C_{ks} V(s,t)$. 
Then, the game with the utilities $U_{ik}$ is an exact potential game with potential
\begin{align}
\label{eq:potential}
    W(\va,t)
:= \sum_{i=1}^N \bar V_{a_i}(t)
   - \frac{\rho}{2}\sum_{i\neq j} C_{a_i,a_j}.
\end{align}
Consequently, $\va^\star$ is a projected Nash equilibrium if and only if
it is a local maximizer of  $W(\cdot,t)$.
\end{lemma}

\begin{proof}
By definition,
$U_{ik}(\va_{-i},t)
= \sum_{s\in\mathcal A} C_{ks}V(s,t)
   - \rho\sum_{j\neq i}\sum_{s\in\mathcal A} C_{ks}\mathbf 1[a_j=s]
   = \bar V_k(t)-\rho\sum_{j\neq i} C_{k,a_j}$.
Consider a unilateral deviation of agent $i$ from $a_i$ to $a_i'$.
Since $C$ is symmetric,
$\Delta  W
:=  W(a_i',a_{-i},t)- W(a_i,a_{-i},t) 
= \bar V_{a_i'}(t)-\bar V_{a_i}(t)
   - \rho\sum_{j\neq i}\bigl(C_{a_i',a_j}-C_{a_i,a_j}\bigr)
   =U_i(a_i',\va_{-i},t)-U_i(a_i,\va_{-i},t)$.
Therefore $W$ is an exact potential for the projected game.
The equivalence between projected Nash equilibria and local maximizers of $W$ is the standard characterization for finite exact
potential games~\cite{monderer1996potential}.
\end{proof}

\begin{proposition}[Convergence to projected BR and NE]
\label{prop:NE}
Let $\Gamma<0$, $A^b_{k_\star}\neq 0$, and suppose \eqref{eq:suppression} together with 
Assumptions~\ref{ass:dominant_mode} and~\ref{ass:local_regime} hold. Suppose agents update asynchronously. 
Then, for any non-zero initial condition,
% in the basin of attraction of a non-zero stable branch, 
the dynamics~\eqref{eq:nod} converge 
% locally exponentially at rate $\min(1,2|\mu|)$ 
to a geometry-aware best response $\mathrm{BR}^A(t)$ of the coverage game at each $t$. Consequently, successive strict projected-best-response updates converge to a
projected Nash equilibrium of the coverage game.
\end{proposition}
\begin{proof}
% At each $t$, we establish convergence to a best response in amplitude and angle separately, with overall rate $\min(1,2|\mu|)$: $\alpha>\alpha_c$ and $\Gamma<0$ give stable branch $r^\star=\sqrt{-\mu/\Gamma}$ with $\partial_r g_r(r^\star)=-2|\mu|<0$ (Section~\ref{subsec:supercritical}), so the radial component converges exponentially at rate $2|\mu|$. $A^b_{k_\star}\neq 0$ gives a unique stable angular solution $\theta^\star=\theta^b_{k_\star}(t)$ for any $\theta(0)$, since $\partial_\theta g_\theta(\theta^b_{k_\star}) = -1$ globally at leading order, yielding exponential convergence at rate $1$. 
Under Assumption~~\ref{ass:dominant_mode} and~\ref{ass:local_regime}, $\alpha>\alpha_c$ and $\Gamma<0$ yields the radial stable branch $r^\star=\sqrt{-\mu/\Gamma}$ (Section~\ref{subsec:supercritical}). The radial component converges exponentially to this branch from any initial condition in its basin of attraction. For fixed evidence with
$A^b_{k_\star}\neq 0$, Equation~\eqref{eq:polar_theta} admit the unique stable equilibrium $\theta^\star=\theta^b_{k_\star}(t)$ for any $\theta(0)$.
Since $\tau_b\gg\tau_z$, the evidence $\vb_i(t)$ is quasi-static on the decision timescale, allowing~\eqref{eq:nod} to track the moving attractor adiabatically.
By Section~\ref{subsec:proj_BR} and Lemma~\ref{lem:suppression}, the attractor $(r^\star,\theta^\star)$ selects $a_i^\star = \mathrm{BR}_i^A(t)$.
% , the geometry-aware best response of agent $i$ at each $t$. 
% The overall convergence rate is $\min(1,2|\mu|)$, determined by the slower of the two modes.
\end{proof}

\subsection{Numerical Experiment}
\label{subsec:numerical}

We simulate a repeated coverage game over 700 steps: $N\!=\!10$ agents repeatedly choose actions from an action set of $K\!=\!18$ sectors on a \hs{circle}. %We simulate a repeated coverage game with $N=10$ agents on a ring of $K=18$ sectors over $T=700$ steps. 
The coupling matrix $A$ is sampled from a symmetric circulant Mexican-hat kernel (Figure~\ref{fig:sim}(B)), has dominant eigenmode $k_\star=1$ (Assumption~\ref{ass:dominant_mode}) and encodes 
% the goal of positioning each agent at the event density center of mass of its neighborhood. 
\hs{geometry-aware responses to event density and congestion.}
%so it can respond bidirectionally around each peak. 
The event density $V(k,t)$ has three unequal peaks initially near sectors  $\{3,\,8\text{--}10,\,14\text{--}16\}$ (Figure~\ref{fig:sim}(A)).
\hs{These rotate at $0.01$ sectors per step,} with a transient loss of a clearly dominant peak at $t\approx 280$ and \hs{an additional smooth nine-sector} shift at $t\approx 430$. We set $S = \tanh$, $d_z = 1$, $\tau_z = 0.1$, $\kappa = 0.1$, $\rho = 0.32$, and logit inverse temperature $7$.
% , on timescale $\tau_b\gg\tau_z$. 
The objective is to maximize weighted coverage $W(a,t)$ whose local maximizers are Nash equilibria (Lemma~\ref{lem:potential}).
\hs{We compare subcritical NOD ($\alpha=0.96<\alpha_c=1$) with best response and logit dynamics. NOD recovers $\Psi_A(\vb_i)$ from raw evidence $\vb_i$ alone; both baselines receive the projection directly.}
Figure~\ref{fig:sim}(C) plots the agent \hs{allocations} (white circles), \hs{the potential,} the cumulative switches, and the BR fraction (fraction of agents playing $\mathrm{BR}^A_i(t)$ at the end of each decision-epoch).
% for neuromorphic decision dynamics (NOD), best response, and logit dynamics. 
% Subcritical NOD ($\alpha=0.96 < \alpha_c=1$) is compared against \hs{ best response and} logit dynamics both supplied $\Psi_A(\vb_i)$ directly. NOD recovers $\Psi_A(\vb_i)$ from raw evidence $\vb_i$ alone.
% Figure~\ref{fig:comparison} shows that NOD outperforms logit across all four metrics. NOD agents spontaneously spread into three stable clusters around the dominant peaks of $V(\cdot,t)$, maximizing $W(a,t)$ bidirectionally, while logit agents disperse incoherently across all $18$ sectors. 
% \hs{Figure~\ref{fig:sim}(C) shows that NOD achieves performance comparable to best response: its potential $W(a,t)$ closely tracks that of best response, and both maintain}
% BR fraction $=1.0$ for all $700$ steps (Proposition~\ref{prop:NE}), whereas logit fluctuates between $0$ and $1$ and never stabilizes. 
% Subcritical hysteresis (Proposition~\ref{prop:hysteresis}) renders NOD immune to the ambiguity window near $t\approx 280$ which makes the logit dynamics BR fraction collapse repeatedly. At $t\approx 430$, NOD executes a single deterministic reallocation when the directional shift exceeds $\delta^\star$ and immediately re-commits, accumulating $O(10)$ total switches over $700$ steps versus $O(5000)$ for logit.
\hs{
NOD closely tracks the potential achieved by best response, and maintains BR fraction close to $1.0$ throughout (Proposition~\ref{prop:NE}). NOD's persistent allocations follow the rotating event-density landscape and adapt to the directional shift, consistent with persistent but adaptive decisions in the subcritical regime (Proposition~\ref{prop:hysteresis}). Logit agents remain dispersed, with lower projected potential and a BR fraction that continually fluctuates between $0$ and $1$. NOD and best response each accumulate $\mathcal O(100)$ total switches, compared with $\mathcal O(6000)$ for logit.}

\begin{figure}[tpb]
    \centering
    \includegraphics[width=0.9\linewidth]{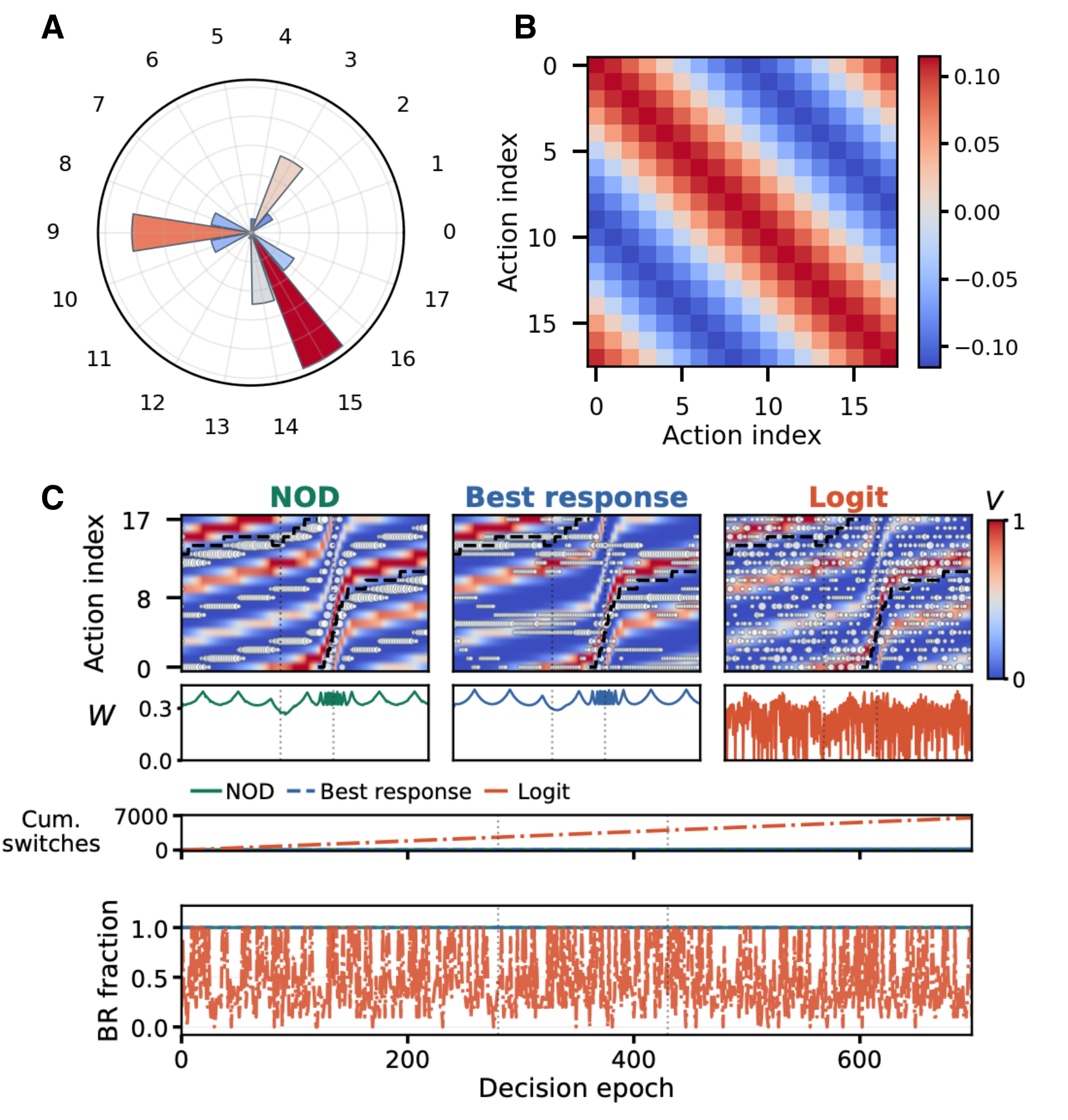}
    \caption{(A) Raw event density $V(k,0)$ on the action ring.
  (B) Mexican-hat action-coupling matrix $A$ with dominant eigenmode $k_\star=1$. (C) Neuromorphic decision dynamics with subcritical NOD, Best response, and logit dynamics on a coverage task \hs{ with slowly rotating event density}. \textit{Top:} Heatmap shows event density $V(k,t)$; white circles area indicates population fraction at each action. 
  \hs{The black dashed line marks the peak of the projected event density $\Psi_A(V(\cdot,t))$}
  \hs{Below each heatmap, the potential $W(a,t)$ is shown. NOD and best response sustain similar values of $W$, both higher than logit.}
  \hs{Vertical dotted lines mark the amplitude-ambiguity window centered at $t=280$ and the additional directional shift in $V(\cdot,t)$ centered at $t=430$.}
  % The strip below each heatmap shows the modal action (red, largest population fraction) and effective utility (black). 
  Neuromorphic agents commit to dominant peaks and track them stably; logit agents scatter without committing. \textit{Middle:} Cumulative switches over 700 steps.
  % : \hs{$\mathcal{O}(100)$ for NOD and best response, and $\mathcal{O} (6000)$ for logit.} 
  \textit{Bottom:} BR fraction remains close to 1 for NOD while logit fluctuates.}
    \label{fig:sim}
     \vspace{-1.5em}
\end{figure}

\section{Conclusion}
\label{sec:conclusion}
We developed a neuromorphic realization of best response for finite-action games with relational action structure. Unlike classical best response rules, which specify only the selected action, or logit models, which specify a response distribution, our approach models decision making as a nonlinear dynamical process whose stable attractors realize committed actions. This mechanistic viewpoint explains how commitment forms, how near-symmetric alternatives are resolved through basins of attraction, and why decisions persist under transient perturbations. For symmetric circulant coupling, the coupling operator selects the evidence components governing decision formation, yielding a geometry-aware best response with exponential convergence, perturbation robustness, and hysteretic commitment in the subcritical regime. The coverage game instantiation showed that logit dynamics is prone to oscillatory switching even with geometry-aware utility, whereas NOD avoids oscillations and recovers Nash equilibria. These properties are especially relevant for resource-constrained robotic systems, where coordination must be achieved with limited computation, noisy observations, and little tolerance for indecisive switching.
 
%%%%%%%%%%%%%%%%%%%%%%%%%%%%%%%%%%%%%%%%%%%%%%%%%%%%%%%%%%%%%%%%%%%%%%%%%%%%%%%%

%%%%%%%%%%%%%%%%%%%%%%%%%%%%%%%%%%%%%%%%%%%%%%%%%%%%%%%%%%%%%%%%%%%%%%%%%%%%%%%%
\section{ACKNOWLEDGMENTS}
We thank Giovanna Amorim, Ian Xul Belaustegui, and Anastasia Bizyaeva for helpful discussions. % during the development of this \hs{work}.

%%%%%%%%%%%%%%%%%%%%%%%%%%%%%%%%%%%%%%%%%%%%%%%%%%%%%%%%%%%%%%%%%%%%%%%%%%%%%%%%

% %%%%%%%%%%%%%%%%%%%%%%%%%%%%%%%%%%%%%%%%%%%%%%%%%%%%%%%%%%%%%%%%%%%%%%%%%%%%%%%%
\bibliographystyle{IEEEtran}
\bibliography{ref}

\end{document}